\documentclass[a4paper,reqno,11pt, oneside]{amsart}

\usepackage{amssymb}
\usepackage{amstext}
\usepackage{amsmath}
\usepackage{amscd}
\usepackage{amsthm}
\usepackage{amsfonts}
\usepackage{enumerate}
\usepackage{graphicx}
\usepackage{tikz}
\usepackage{tikz-cd}
\usepackage{color}
\usepackage{microtype}
\usepackage{placeins}
\usetikzlibrary{arrows.meta}
\numberwithin{equation}{section}

\newcommand{\calC}{\mathcal{C}}

\newcommand{\calF}{\mathcal{F}}

\newcommand{\calO}{\mathcal{O}}
\newcommand{\calP}{\mathcal{P}}
\newcommand{\calQ}{\mathcal{Q}}

\newcommand{\ZZ}{\mathbb{Z}}

\newcommand{\RR}{\mathbb{R}}
\newcommand{\CC}{\mathbb{C}}

\newcommand{\eb}{\mathbf{e}}

\newcommand{\Hom}{\operatorname{Hom}}

\def\opn#1#2{\def#1{\operatorname{#2}}}
\opn\conv{conv} \opn\mut{mut} \opn\GL{GL} \opn\cone{cone} \opn\rec{rec}

\newtheorem{thm}{Theorem}[section]

\newtheorem{lemma}[thm]{Lemma}
\newtheorem{prop}[thm]{Proposition}

\theoremstyle{definition}
\newtheorem{defi}[thm]{Definition}
\newtheorem{ex}[thm]{Example}
\newtheorem{nota}[thm]{Notation}

\theoremstyle{remark}
\newtheorem{rem}[thm]{Remark}

\begin{document}

\title{Two poset polytopes are mutation-equivalent}
\author{Akihiro Higashitani} 

\address{Department of Pure and Applied Mathematics, Graduate School of Information Science and Technology, Osaka University, Suita, Osaka 565-0871, Japan}
\email{higashitani@ist.osaka-u.ac.jp}


\begin{abstract} 
The combinatorial mutation $\mut_w(P,F)$ for a lattice polytope $P$ was introduced in the context of mirror symmetry for Fano manifolds in \cite{ACGK}. 
It was also proved in \cite{ACGK} that for a lattice polytope $P \subseteq N_\RR$ containing the origin in its interior, 
the polar dual $P^* \subseteq M_\RR$ and $\mut_w(P,F)^* \subseteq M_\RR$ have the same Ehrhart quasi-polynomial. 
To extend this framework, we introduce combinatorial mutation for rational pointed polyhedra in $N_\RR$ containing the origin in their interiors. Such polyhedra are Minkowski sums of rational polytopes and rational polyhedral pointed cones. On the dual side $M_\RR$, the construction applies to full-dimensional rational polytopes containing the origin, not necessarily in their interiors. 
As an application of this extension of the combinatorial mutation, we prove that 
the chain polytope of a poset $\Pi$ can be obtained by a sequence of combinatorial mutations in $M_\RR$ from the order polytope of $\Pi$. 
Namely, the order polytope and the chain polytope of the same poset $\Pi$ are mutation-equivalent. 
\end{abstract}

\maketitle

\FloatBarrier
\section{Introduction}

The key notion of the present paper is the \textit{combinatorial mutation} for polyhedra. 
Originally, in \cite{ACGK}, the notion of combinatorial mutation for lattice polytopes was introduced from the viewpoint of mirror symmetry for Fano manifolds. 
Fano manifolds are one of the most important objects in algebraic geometry and their classification is of particular interest. 
One motivation for introducing combinatorial mutation in \cite{ACGK} was to study this classification through mirror symmetry. 
It is expected that every Fano manifold corresponds to a certain Laurent polynomial, called a mirror partner, via mirror symmetry (see \cite{CCGGK}). 
We say that a Laurent polynomial $f \in \CC[x_1^{\pm1},\ldots,x_n^{\pm1}]$ is a \textit{mirror partner} of an $n$-dimensional Fano manifold $X$ 
if the period $\pi_f$ of $f$ coincides with the quantum period $\hat{G}_X$ of $X$. 
(See, e.g., \cite{ACGK} or \cite{CCGGK} or the references therein for more details.) 
It is also expected that a Fano manifold $X$ admits a toric degeneration $X_P$, 
where $P$ is the Newton polytope of a Laurent polynomial $f$ which is a mirror partner of $X$ and $X_P$ denotes the toric variety associated to the lattice polytope $P$. 
On the other hand, Laurent polynomials having the same period are not unique, 
so the same Fano manifold may admit several such toric degenerations. 
To understand the relations among these degenerations, mutations of Laurent polynomials were introduced (see \cite{ACGK} and \cite{GU}). These are birational transformations analogous to cluster transformations. 
Namely, it was shown in \cite[Lemma 1]{ACGK} that for Laurent polynomials $f,g$, the period of $f$ and that of $g$ are equal if $g$ can be obtained by mutations from $f$. 
Combinatorial mutation describes the corresponding operation on Newton polytopes. 
See Section~\ref{sec:N} for the precise definition of combinatorial mutation, in particular Subsection~\ref{subsec:polytope} for polytopes. 
Note that combinatorial mutation in the sense of \cite{ACGK} is the operation for the lattice polytopes in $N_\RR$ (not in $M_\RR$). 

\smallskip

\begin{nota}
Let us fix our notation used throughout this paper. We employ the usual notation used in the context of toric geometry. 
Let $N \cong \ZZ^d$ be a lattice of rank $d$, let $M=\Hom(N,\ZZ) \cong \ZZ^d$, $N_\RR=N \otimes_\ZZ \RR$ and $M_\RR=M \otimes_\ZZ \RR$. 
We denote the natural pairing between $N_\RR$ and $M_\RR$ by $\langle \cdot, \cdot \rangle : M_\RR \times N_\RR \rightarrow \RR$. 

Given $X \subseteq N_\RR$ (resp. $Y \subseteq M_\RR$), let $\conv(X) \subseteq N_\RR$ (resp. $\conv(Y) \subseteq M_\RR$) 
denote the convex hull of $X$ (resp. $Y$), and let $\cone(X) \subseteq N_\RR$ (resp. $\cone(Y) \subseteq M_\RR$) 
denote the cone generated by $X$ (resp. $Y$). 
\end{nota}

\smallskip

Combinatorial mutation preserves certain invariants associated with lattice polytopes. 
For example, let $P \subseteq N_\RR$ be a lattice polytope containing the origin in its interior, 
and let $P'=\mut_w(P,F) \subseteq N_\RR$ be the lattice polytope obtained from $P$ by a combinatorial mutation with respect to $w$ and $F$. 
Note that $P'$ is again a lattice polytope containing the origin in its interior. 
Then the Ehrhart quasi-polynomial of the polar dual $P^* \subseteq M_\RR$ coincides with that of $P'^* \subseteq M_\RR$, 
although the two dual polytopes need not be unimodularly equivalent. 
(See the definition of $(-)^*$ in Subsection~\ref{subsec:prepare}.) 
Moreover, we can directly describe $P'^* \subseteq M_\RR$ from $P^* \subseteq M_\RR$ in terms of the operation in $M_\RR$. 
More concretely, we have $P'^*=\varphi_{w,F}(P^*)$, where $\varphi_{w,F}:M_\RR \rightarrow M_\RR$ is a piecewise-linear tropical map (see Definition~\ref{def:tropical}). 
However, this original  framework concerns polar duals of lattice polytopes containing the origin in their interiors. 
If a full-dimensional rational polytope $Q\subseteq M_\RR$ contains the origin on its boundary, then $Q^*\subseteq N_\RR$ is unbounded. 
The motivation of this paper is to extend this framework to full-dimensional rational polytopes in $M_\RR$ containing the origin. 
More precisely, we introduce the class $\calP_N$, whose members are precisely the polar duals of the polytopes just described,  
in Subsection~\ref{subsec:prepare} and define combinatorial mutation for this class in Subsection~\ref{subsec:PN}. 
This extension is relevant to our application because both the order polytope and the chain polytope of posets contain the origin on their boundaries. 
Hence, an interpretation of their mutations within the same framework requires unbounded polyhedra on the $N_\RR$-side. 

The main result of this paper is an application of this extension to two poset polytopes, the order polytope and the chain polytope. 
Introduced by Stanley~\cite{S86}, the order polytope and the chain polytope encode poset filters and antichains, respectively,  
and relate the enumeration of order-preserving maps to lattice-point counting. 
Stanley's transfer map shows that they have the same Ehrhart polynomial, although they need not be unimodularly equivalent. 
We prove that this transfer map factors into elementary combinatorial mutations with a convex polytope at every intermediate step (Theorem~\ref{main}). 
Thus, our result refines the known Ehrhart equivalence by placing the two polytopes in the same mutation-equivalence class. See Section~\ref{sec:poset} for more details.

Fujita and the author extended this approach to certain marked order and marked chain polytopes, 
proving their combinatorial mutation-equivalence after suitable lattice translations \cite[Theorem~5.3]{FH}. 
This includes the Gelfand--Tsetlin and Feigin--Fourier--Littelmann--Vinberg polytopes in types $A$ and $C$ \cite[Theorems~5.5 and~5.6]{FH}, 
relating the construction studied here to Newton--Okounkov bodies of flag varieties. 

The paper is organized as follows. In Section~\ref{sec:N}, we recall some material from convex geometry 
and extend combinatorial mutation from rational polytopes to rational polyhedral pointed cones and then to $\calP_N$. 
Section~\ref{sec:M} describes the dual operation and its compatibility with polar duality. The polytopes in $\calQ_M$ are in bijection with the polyhedra in $\calP_N$ via polar duality. 
Section~\ref{sec:poset} proves the mutation-equivalence of order and chain polytopes. 

Finally, it should be mentioned that according to Alexander Kasprzyk, the authors of \cite{ACGK} also thought of the extension of the framework of combinatorial mutation 
to rational polyhedral cones, and so did Daniel Cavey and Thomas Hall. 
Since the authors of \cite{ACGK} needed to treat the combinatorial mutation only for lattice polytopes in the context of mirror symmetry for Fano manifolds, 
they introduced it only for lattice polytopes. 
Although the author heard the overview of this idea from Alexander Kasprzyk, the precise proofs in the present paper are given by the author himself. 

\bigskip

\subsection*{Acknowledgements}
The author thanks Alexander Kasprzyk for valuable lectures and discussions on combinatorial mutations of polytopes 
and thanks Daniel Cavey and Thomas Hall for stimulating discussions with them. 
The author also thanks Naoki Fujita for his helpful comments on the paper. The author learned a lot from him, especially on cluster mutation.

The author is partially supported by JSPS KAKENHI Grant Numbers JP20K03513 and JP24K00521. 

Finally, the author thanks the anonymous referees for their careful reading and valuable suggestions, 
which have improved the clarity and presentation of the paper. 

\bigskip

\FloatBarrier
\section{Extension of Combinatorial Mutation}\label{sec:N}

In this section, we extend combinatorial mutation to the class $\calP_N$ of rational pointed polyhedra containing the origin in their interiors. We first treat rational polytopes and rational polyhedral pointed cones, and then combine the two constructions. 

\subsection{Preparation from convex geometry}\label{subsec:prepare}
We first recall some basic facts from convex geometry. See, e.g., \cite{Sch} for an introduction. 

We consider the family of polyhedra (not necessarily convex polytopes) which are of the following form: 
\begin{align*}
\calP := \left\{ \bigcap_{v \in S} H_{v,\geq -1} \cap \bigcap_{v' \in T}H_{v',\geq 0} \subseteq N_\RR \ \middle| \ S,T \subseteq M_\RR,\; |S|,|T| < \infty \right\}; 
\end{align*}
where $H_{v,\geq k}=\{u \in N_\RR \mid \langle v,u \rangle \geq k\}$ for $v \in M_\RR$ and $k \in \RR$. 
Note that every $P \in \calP$ is non-empty since $P$ always contains the origin of $N_\RR$. 
Similarly, we define $\calQ$ by swapping the roles of $N_\RR$ and $M_\RR$. 
Namely, all convex sets in $\calP$ sit in $N_\RR$ and those in $\calQ$ sit in $M_\RR$, but the objects are essentially the same. 
Note that a convex polytope in $N_\RR$ (resp. $M_\RR$) containing its origin belongs to $\calP$ (resp. $\calQ$), 
but a convex polytope not containing the origin does not belong to $\calP$ (resp. $\calQ$). 
We work with polyhedra containing the origin in order to use polar duality. 

We say that a convex polytope $P$ in $N_\RR$ (resp. $M_\RR$) is a \textit{lattice polytope} if all of its vertices are points in $N$ (resp. $M$), 
and we say that $P$ is a \textit{rational polytope} if there is a positive integer $n$ such that $nP$ is a lattice polytope. 
We call a polyhedral cone $C$ in $N_\RR$ (resp. $M_\RR$) \textit{rational} if the generators of $C$ can be chosen from $N$ (resp. $M$). 
We call a polyhedral cone $C$ \textit{pointed} if $C$ contains no $1$-dimensional linear subspace. 

Given $P \in \calP$, we consider the \textit{polar dual} $P^* \subseteq M_\RR$ of $P$ defined as 
\[P^* := \{v \in M_\RR\mid\langle v,u\rangle\ge-1 \text{ for all } u\in P\}\subseteq M_\RR. \]
We can also define $Q^* \subseteq N_\RR$ from $Q \in \calQ$ in the same way. 
\begin{prop}[{\cite[Theorem 9.1]{Sch}}]\label{prop:dual}
We have the following statements: 
\begin{itemize}
\item[(i)] $P^* \in \calQ$ for any $P \in \calP$, and vice versa; 
\item[(ii)] $(P^*)^*=P$ for any $P \in \calP$ and $(Q^*)^*=Q$ for any $Q \in \calQ$. 
\end{itemize}
\end{prop}
The same statement appears in \cite[Proposition~6.5]{HN}, so we omit its proof.

In the following sections, we treat the classes $\calP_N \subseteq \calP$ and $\calQ_M \subseteq \calQ$ which are defined as follows: 
\begin{itemize}
\item Let $\calQ_M$ be the set of full-dimensional rational polytopes $Q$ in $M_\RR$ containing the origin (not necessarily in the interior). 
\item Let $\calP_N=\{Q^* \mid Q \in \calQ_M\}$. 
\end{itemize}
Note that every polyhedron $P \in \calP_N$ contains the origin in its interior by Proposition~\ref{prop:origin}. 
Moreover, for any $P \in \calP_N$, we have the following canonical decomposition: 
\begin{equation}\label{eq:decomp}
\begin{split}
&P=P'+C, \text{ where }P'=\conv(V(P)) \text{ and } \\
&C=\rec(P)=\{y \in N_\RR \mid x+ry \in P \text{ for any }x \in P \text{ and } r \in \RR_{\geq 0}\},  
\end{split}
\end{equation}
where $V(P)$ is the set of vertices of $P$, and $\rec(P)$ is its recession cone; see \cite[Section~8.9]{Sch}. 
The pointedness of $C$ follows from the full-dimensionality of $P^*$ (\cite[Corollary~9.1a(i)]{Sch}). 
Equivalently, $\calP_N$ consists of the rational pointed polyhedra containing the origin in their interiors. 
Figure~\ref{ex:obj} illustrates the relation between the classes $\calP_N$ and $\calP$. 
The left-hand square $[-1,1]^2$ belongs to $\calP_N$, since it is a rational polytope and the origin lies in its interior. 
The middle square $[0,1]^2$ belongs to $\calP$ but not to $\calP_N$, since it contains the origin only on its boundary. 
Finally, the right-hand square $[1,2]^2$ does not belong to $\calP$, since it does not contain the origin. 
\begin{figure}[htb]
\centering
\begin{tikzpicture}[scale=1.04]

\begin{scope}
  \draw[->] (-1.4,0) -- (1.4,0) node[right] {$x$};
  \draw[->] (0,-1.4) -- (0,1.4) node[above] {$y$};

  \draw[thick]
    (-1,-1) -- (1,-1) -- (1,1) -- (-1,1) -- cycle;

  \foreach \p/\pos in {
    (-1,-1)/below left,
    (1,-1)/below right,
    (-1,1)/above left,
    (1,1)/above right}
  {
    \fill \p circle (0.03);
    \node[\pos] at \p {$\p$};
  }
\end{scope}

\begin{scope}[xshift=4.2cm]
  \draw[->] (-0.2,0) -- (1.4,0) node[right] {$x$};
  \draw[->] (0,-0.2) -- (0,1.4) node[above] {$y$};

  \draw[thick]
    (0,0) -- (1,0) -- (1,1) -- (0,1) -- cycle;

  \foreach \p/\pos in {
    (0,0)/below left,
    (1,0)/below,
    (0,1)/left,
    (1,1)/above right}
  {
    \fill \p circle (0.03);
    \node[\pos] at \p {$\p$};
  }
\end{scope}

\begin{scope}[xshift=8.4cm, yshift=-1cm]
  \draw[->] (-0.2,0) -- (2.4,0) node[right] {$x$};
  \draw[->] (0,-0.2) -- (0,2.4) node[above] {$y$};

  \draw[thick]
    (1,1) -- (2,1) -- (2,2) -- (1,2) -- cycle;

  \foreach \p/\pos in {
    (1,1)/below left,
    (2,1)/below right,
    (1,2)/above left,
    (2,2)/above right}
  {
    \fill \p circle (0.03);
    \node[\pos] at \p {$\p$};
  }
\end{scope}

\end{tikzpicture}
\caption{Examples and non-examples of objects in $\calP$ and $\calP_N$}\label{ex:obj}
\end{figure}

We recall the following useful description of the polar dual. 
Let $P \in \calP_N$ and let $P=P'+C$ be the decomposition in \eqref{eq:decomp}. 
Let $P'=\conv(\{u_1,\ldots,u_p\})$ and let $C=\cone(\{u_1',\ldots,u_q'\})$ for $u_i,u_j' \in N_\RR$. 
Then $P^*$ is given by \[P^*=\bigcap_{i=1}^p H_{u_i,\geq -1} \cap \bigcap_{j=1}^qH_{u_j',\geq 0}.\]

\begin{ex}\label{ex:running1}
Let $d=2$ and let us regard both $N_\RR$ and $M_\RR$ as $\RR^2$. 
Let $$P=H_{(1,0), \geq -1} \cap H_{(0,1), \geq -1} \cap H_{(1,1), \geq -1} \subseteq \RR^2.$$
Then $P$ can be written as the Minkowski sum of $P'=\conv(\{(-1,0),(0,-1)\})$ and $C=\cone(\{(1,0),(0,1)\})$. 
Thus, $P^*$ can be written as follows: 
$$P^*=H_{(-1,0), \geq -1} \cap H_{(0,-1), \geq -1} \cap H_{(1,0), \geq 0} \cap H_{(0,1), \geq 0} \subseteq \RR^2.$$
Equivalently, we have $P^*=\conv(\{(0,0),(1,0),(0,1),(1,1)\})$. See Figure~\ref{fig:duals}. 
\begin{figure}[htb]
\centering
\begin{tikzpicture}[scale=1.5]

\begin{scope}[xshift=-0.5cm,scale=0.9]

\begin{scope}
  \clip (-1.5,-1.5) rectangle (1.8,1.8);

  \fill[blue!20,opacity=0.7]
    (-1,-1) -- (-1,1.8) -- (1.8,1.8) -- (1.8,-1) -- cycle;

  \clip (-1,0) -- (0,-1) -- (1.8,-1) -- (1.8,1.8) -- (-1,1.8) -- cycle;
\end{scope}

\fill[white]
  (-1,-1) -- (-1,0) -- (0,-1) -- cycle;

\draw[->] (-1.5,0) -- (1.6,0);
\draw[->] (0,-1.5) -- (0,1.6);

\draw[thick] (-1,0) -- (-1,1.6);
\draw[thick] (0,-1) -- (1.6,-1);
\draw[thick] (-1,0) -- (0,-1);

\fill (-1,0) circle (0.03);
\fill (0,-1) circle (0.03);

\node[left]  at (-1,0.15) {$(-1,0)$};
\node[below] at (0.4,-1) {$(0,-1)$};

\node at (0.6,1.55) {{\Large$P$}};

\end{scope}

\begin{scope}[xshift=3.2cm,yshift=-0.5cm,scale=1.2]

\fill[blue!20,opacity=0.7]
  (0,0) -- (1,0) -- (1,1) -- (0,1) -- cycle;

\draw[->] (-0.3,0) -- (1.4,0);
\draw[->] (0,-0.3) -- (0,1.4);

\draw[thick]
  (0,0) -- (1,0) -- (1,1) -- (0,1) -- cycle;

\foreach \p in {(0,0),(1,0),(0,1),(1,1)}
  \fill \p circle (0.03);

\node[below left] at (0,0) {$(0,0)$};
\node[below]      at (1,0) {$(1,0)$};
\node[left]       at (0,1) {$(0,1)$};
\node[right]      at (1,1) {$(1,1)$};

\node at (0.6,1.35) {{\Large$P^*$}};

\end{scope}

\end{tikzpicture}
\caption{Description of $P$ and $P^*$}\label{fig:duals}
\end{figure}
\end{ex}

\begin{prop}[{\cite[Corollary 9.1a (iii)]{Sch}}]\label{prop:origin}
Let $Q \in \calQ$. 
Then $Q$ is a convex polytope if and only if $Q^* \in \calP$ is a polyhedron containing the origin in its interior. 
\end{prop}

We state the notion of mutation-equivalence before developing the construction. The operations $\mut_w$ and $\varphi_{w,F}$ appearing below are defined in Subsection~\ref{subsec:PN} and Definition~\ref{def:tropical}, respectively. 
\begin{defi}[Combinatorial mutation-equivalence]
Given two polyhedra $P,P' \in \calP_N$ (resp. rational polytopes $Q,Q' \in \calQ_M$), we write $P \sim_N P'$ (resp. $Q \sim_M Q'$) 
if $P'=\mut_w(P,F)$ (resp. $Q'=\varphi_{w,F}(Q)$) for a primitive $w\in M$ and a lattice polytope $F\subseteq w^\perp$, whenever the indicated operation is defined and its image belongs to the stated class. 

We say that two polyhedra $P, P' \in \calP_N$ are \textit{combinatorially mutation-equivalent in $N_\RR$} 
if there exist polyhedra $P=P_0,P_1,\ldots,P_m=P' \in \calP_N$ such that $P_0 \sim_N P_1 \sim_N \cdots \sim_N P_m$. 

Similarly, we say that two rational polytopes $Q,Q' \in \calQ_M$ are \textit{combinatorially mutation-equivalent in $M_\RR$} 
if there exist rational polytopes $Q=Q_0,Q_1,\ldots,Q_\ell=Q' \in \calQ_M$ such that $Q_0 \sim_M Q_1 \sim_M \cdots \sim_M Q_\ell$. 
\end{defi}

\subsection{Combinatorial mutation for rational polytopes}\label{subsec:polytope}

To introduce combinatorial mutation for $P\in\calP_N$, we first treat \textit{rational polytopes}. 
The construction is similar to the approach of \cite[Section~3]{ACGK}, but uses the full sections $P\cap H_{w,h}$ at real heights 
rather than the convex hulls of lattice points at integral heights. 
We also allow a rational factor and do not require the width vector to be primitive. 
Consequently, the convex-geometric arguments below do not use integrality, whereas the lattice assumption on the factor will be restored for Ehrhart invariance in Section~\ref{sec:M}.

Let $P\subseteq N_\RR$ be a rational polytope, and take $w\in M$. For $h\in\RR$, define
\[ H_{w,h}:=\{x\in N_\RR \ \mid \ \langle w,x\rangle=h\},  \qquad P_{w,h}:=P\cap H_{w,h}. \]
We write $w^\perp=H_{w,0}$ and denote the vertex set of $P$ by $V(P)$.
We use the convention $A+\emptyset=\emptyset+A=\emptyset$.

\begin{defi}[{Combinatorial mutation for rational polytopes, cf. \cite[Definition~5]{ACGK}}]\label{def:mutation}
Let $0 \neq w\in M$, and let $F\subseteq w^\perp$ be a nonempty rational polytope.
Suppose that, for each $h\leq0$, there is a possibly empty polytope $G_h\subseteq H_{w,h}$, rational whenever $h$ is rational, such that 
\begin{equation}\label{eq:condition_rat}
 V(P)\cap H_{w,h}\subseteq G_h+(-h)F\subseteq P_{w,h}, 
\end{equation}
where $(-h)F$ denotes dilation by $-h\geq0$. Define 
\[ \mut_w(P,F;\{G_h\}) :=\conv\left(\bigcup_{h\leq0}G_h  \ \cup\ \bigcup_{h\geq0}(P_{w,h}+hF)\right).\]
Proposition~\ref{prop:G_h} shows that this is a rational polytope and is independent of the choice of $\{G_h\}$. 
Therefore, we write $\mut_w(P,F)$ and call it a \textit{combinatorial mutation} of $P$, with \textit{width vector} $w$ and \textit{factor} $F$. 
\end{defi}

There are only finitely many nonpositive heights containing vertices of $P$, and one may set $G_h=\emptyset$ at all other heights. 
Note also that $H_{w,-h}=H_{-w,h}$ and $\dim F\leq d-1$.

\begin{ex}\label{ex:short}
Let us consider an easy example and perform a combinatorial mutation. Let $P=\conv(\{(1,1),(-1,1),(-1,-1),(1,-1)\})$, let $w=(0,1)$ and let $F=\conv(\{(0,0),(1,0)\})$. 
Then $w^\perp=\{(x,0) \mid x \in \RR\}$, so $F\subseteq w^\perp$. 
We seek a family $\{G_h\}$ satisfying \eqref{eq:condition_rat}. For $h\leq0$, the set $V(P)\cap H_{w,h}$ is non-empty only when $h=-1$; 
we may set $G_h=\emptyset$ for $h \leq 0$ with $h \neq -1$. 
If $G_{-1}=\conv(\{(-1,-1),(0,-1)\})$, then \begin{align*} 
V(P) \cap H_{w,-1}&=\{(-1,-1),(1,-1)\}, \\
G_{-1}+F&=\conv\{(-1,-1),(0,-1)\}+\conv\{(0,0),(1,0)\}\\
&=\conv\{(-1,-1),(1,-1)\}=P_{w,-1}. 
\end{align*}
Thus, the condition \eqref{eq:condition_rat} is satisfied. Then 
\begin{align*}
\mut_w(P,F)&=\conv\left(\bigcup_{h \leq 0}G_h \cup \bigcup_{h \geq 0}(P_{w,h}+hF)\right)=\conv\left(G_{-1} \cup (P_{w,1}+F)\right)\\
&=\conv(\{(-1,-1),(0,-1)\} \cup \{(-1,1),(2,1)\}) \subseteq N_\RR.
\end{align*}
\begin{figure}[htb]
\centering
\begin{tikzpicture}[scale=0.8]

\begin{scope}
\draw[->,very thick] (-1.7,-.4)--(-1.7,.65) node[above] {$w$};
\draw[ultra thick] (0,0)--(1,0) node[midway,above] {$F$};

\draw[->] (-2.2,0) -- (2.2,0);
\draw[->] (0,-2.2) -- (0,2.2);

\draw[thick]
  (1,1) -- (-1,1) -- (-1,-1) -- (1,-1) -- cycle;

\foreach \p in {(1,1),(-1,1),(-1,-1),(1,-1)}
  \fill \p circle (0.06);

\node[above right, font=\scriptsize] at (1,1) {$(1,1)$};
\node[above left,  font=\scriptsize] at (-1,1) {$(-1,1)$};
\node[below left,  font=\scriptsize] at (-1,-1) {$(-1,-1)$};
\node[below right, font=\scriptsize] at (1,-1) {$(1,-1)$};

\end{scope}

\begin{scope}[xshift=6cm]

\draw[->] (-2.2,0) -- (2.5,0);
\draw[->] (0,-2.2) -- (0,2.2);

\draw[thick]
  (2,1) -- (-1,1) -- (-1,-1) -- (0,-1) -- cycle;

\foreach \p in {(2,1),(-1,1),(-1,-1),(0,-1)}
  \fill \p circle (0.06);

\node[above right, font=\scriptsize] at (2,1) {$(2,1)$};
\node[above left,  font=\scriptsize] at (-1,1) {$(-1,1)$};
\node[below left,  font=\scriptsize] at (-1,-1) {$(-1,-1)$};
\node[below right, font=\scriptsize] at (0,-1) {$(0,-1)$};
\end{scope}
\end{tikzpicture}
\caption{An example of combinatorial mutation}\label{fig:square}
\end{figure}

The combinatorial mutation is shown in Figure~\ref{fig:square}. With the same $P,w,F$, one may instead take $G_0=\{(0,0)\}$, leaving the other slices unchanged;
the resulting mutation is the same. On the other hand, by replacing $F$ by $\conv(\{(0,0),(3,0)\})$, we see that the combinatorial mutation is not defined. 
In fact, at height $-1$, any non-empty sum $G_{-1}+F$ has horizontal width at least $3$, whereas $P_{w,-1}$ has width $2$. 
Thus, not every triple $(P,w,F)$ admits a combinatorial mutation. 
\end{ex}
Intuitively, combinatorial mutation removes a multiple of $F$ from each negative-height slice and adds a multiple of $F$ to each positive-height slice. 
The following lemma makes this description precise and proves reversibility. 
\begin{lemma}[{cf. \cite[Lemma 2]{ACGK}}]\label{lem:1}
Let $Q=\mut_w(P,F;\{G_h\})$. Then we have \[\mut_{-w}(Q,F;\{P_{-w,h}\})=P.\] 
\end{lemma}
\begin{proof}
First, we show that
\begin{align}
 Q_{w,h}&=P_{w,h}+hF &&\text{for }h\geq 0, \label{eq:positive-slice}\\
 Q_{w,h}+(-h)F&\subseteq P_{w,h} &&\text{for }h\leq 0.
 \label{eq:negative-slice}
\end{align}
Take $q\in Q_{w,h}$. By definition of $Q$, we may write $q=\sum_{i=1}^r\lambda_iq_i$ and $h=\sum_{i=1}^r\lambda_i h_i$, 
where $\lambda_i\geq0$, $\sum_{i=1}^r \lambda_i=1$, and each $q_i$ belongs either to $G_{h_i}$ with $h_i\leq0$ or to $P_{w,h_i}+h_iF$ with $h_i\geq0$. 
In either case, we may write $q_i=p_i+h_if_i$ with $p_i\in P_{w,h_i}$ and $f_i\in F$. 
When $h_i<0$, we may choose any $f_i\in F$ and then set $p_i=q_i-h_if_i$ since $G_{h_i}+(-h_i)F\subseteq P_{w,h_i}$. 
Set \[ h_+=\sum_{h_i>0}\lambda_ih_i \;\text{ and }\; h_-=-\sum_{h_i<0}\lambda_ih_i,\] i.e., $h=h_+-h_-$. 

Suppose that $h\geq0$. If $h_+>0$, put 
\[
 f_+=\frac{1}{h_+}\sum_{h_i>0}\lambda_ih_if_i\in F,
\]
and, for $h_i<0$, choose $f_i=f_+$. Then
\[
 q=\sum_{i=1}^r\lambda_ip_i+hf_+.
\]
Since $\sum_{i=1}^r\lambda_ip_i\in P_{w,h}$, we obtain
$q\in P_{w,h}+hF$. The case $h_+=0$ is immediate. Thus
$Q_{w,h}\subseteq P_{w,h}+hF$, while the reverse inclusion follows
directly from the definition of $Q$. This proves
\eqref{eq:positive-slice}.

Next, suppose that $h\leq0$ and take $f\in F$. If $h_->0$, let
\[
 f_+=\frac{1}{h_+}\sum_{h_i>0}\lambda_ih_if_i
\]
when $h_+>0$, and choose any $f_+\in F$ when $h_+=0$. Since $h_+\leq h_-$, we have 
\[
 f_-=\frac{h_+}{h_-}f_+ +\left(1-\frac{h_+}{h_-}\right)f \in F. 
\]
By choosing $f_i=f_-$ for all $h_i<0$, we obtain $\sum_{i=1}^r\lambda_i h_i f_i =h_+f_+-h_-f_-=hf$. Hence, 
\[ q+(-h)f=\sum_{i=1}^r\lambda_ip_i\in P_{w,h},\]
which proves \eqref{eq:negative-slice}. The case $h_-=0$ is immediate.

We now consider the combinatorial mutation of $Q$ with respect to $-w$ and $F$.
For $t\leq0$, \eqref{eq:positive-slice} gives
\[ P_{w,-t}+(-t)F =Q_{w,-t}=Q_{-w,t}.\]
Thus the family $\{P_{w,-t}\}_{t\leq0}$ satisfies the condition \eqref{eq:condition_rat} for $(Q,-w,F)$. Therefore,
\[
 \mut_{-w}(Q,F)
 =
 \conv\left(
 \bigcup_{t\leq0}P_{w,-t}
 \,\cup\,
 \bigcup_{t\geq0}(Q_{-w,t}+tF)
 \right).
\]
By \eqref{eq:negative-slice},
\[
 Q_{-w,t}+tF=Q_{w,-t}+tF\subseteq P_{w,-t}
\qquad (t\geq0),
\]
and hence, $\mut_{-w}(Q,F)\subseteq P$.

Conversely, let $v\in V(P)$ and put $h=\langle w,v\rangle$.
If $h\geq0$, then $v\in P_{w,h}\subseteq \mut_{-w}(Q,F)$. If $h\leq0$, then $v\in G_h+(-h)F$ by \eqref{eq:condition_rat}. 
Since $G_h\subseteq Q_{w,h}$, we have 
\[ v\in Q_{w,h}+(-h)F   =Q_{-w,-h}+(-h)F   \subseteq \mut_{-w}(Q,F).\]
Thus, every vertex of $P$ belongs to $\mut_{-w}(Q,F)$, and we conclude $\mut_{-w}(Q,F)=P$. 
\end{proof}

\begin{lemma}[{cf. \cite[Lemma~3]{ACGK}}]\label{lem:2}
Let $Q=\mut_w(P,F;\{G_h\})$. Then
\begin{align*}
 V(Q)&\subseteq\{p+\langle w,p\rangle f:p\in P,\ f\in V(F)\},\\
 V(P)&\subseteq\{q-\langle w,q\rangle f:q\in Q,\ f\in V(F)\}.
\end{align*}
\end{lemma}
\begin{proof}
The second inclusion follows from the first and Lemma~\ref{lem:1}.
Let $v\in V(Q)$ and $h=\langle w,v\rangle$.
If $h>0$, \eqref{eq:positive-slice} implies
$v\in V(P_{w,h}+hF)$, so $v=p+hf$ for some $p\in P_{w,h}$ and
$f\in V(F)$. If $h\leq0$, \eqref{eq:negative-slice} implies $p:=v-hf\in P_{w,h}$ for any $f\in V(F)$.
In either case, $v=p+\langle w,p\rangle f$, as required.
\end{proof}

\begin{prop}[{cf. \cite[Proposition~1]{ACGK}}]\label{prop:G_h}
Let $P\subseteq N_\RR$ be a rational polytope, let $w\in M$,
and let $F\subseteq w^\perp$ be a rational polytope.
Suppose that $(P,w,F)$ satisfies \eqref{eq:condition_rat}
with $\{G_h\}$ and $\{G'_h\}$.
Then both combinatorial mutations are rational polytopes and
\[ \mut_w(P,F;\{G_h\})=\mut_w(P,F;\{G'_h\}).\]
Moreover, for any $u\in M_\RR$, we have
\begin{equation}\label{eq:support-polytope}
 \min_{q\in\mut_w(P,F)}\langle u,q\rangle = \min_{p\in P}\langle u+aw,p\rangle,
\end{equation}
where $a=\min_{f\in F}\langle u,f\rangle$.
\end{prop}
\begin{proof}
Let $Q=\mut_w(P,F;\{G_h\})$ and $Q'=\mut_w(P,F;\{G'_h\})$. 
First, suppose that $G_h=G'_h=\emptyset$ whenever $h\leq 0$ and $V(P)\cap H_{w,h}=\emptyset$. 
Under this assumption, only finitely many $G_h$ and $G'_h$ are nonempty. Moreover, by definition, we have 
\[ \conv\left(\bigcup_{h\geq0}(P_{w,h}+hF)\right) = \conv\left( \bigcup_{f\in V(F)} \{p+\langle w,p\rangle f \mid p\in P\cap H_{w,\geq0}\} \right).\]
For each $f\in V(F)$, the set on the right is a linear image of the rational polytope $P\cap H_{w,\geq0}$, or is empty. 
On the other hand, $G_h$ and $G_h'$ are rational polytopes or empty. Thus, $Q$ and $Q'$ are rational polytopes.

Suppose that $Q\neq Q'$. Without loss of generality, assume that $Q\not\subseteq Q'$.
Since the positive parts in their definitions are the same and $G_0\subseteq P_{w,0}\subseteq Q'$, there exist $h<0$ and $q\in G_h$ with $q\notin Q'$. 
Hence, there is a supporting hyperplane $H_{u,\ell}$ of $Q'$ such that $\langle u,q'\rangle\geq\ell$ holds for all $q'\in Q'$ but $\langle u,q\rangle<\ell$ holds.
Put $a=\min_{f\in F}\langle u,f\rangle$. We claim that \begin{align}\label{eq:PQ} \langle u+aw,p\rangle\geq\ell \;\text{ for all }\; p\in P.\end{align}

It suffices to prove this for $v\in V(P)$. Let $k=\langle w,v\rangle$.
If $k\geq0$, choose $f\in F$ with $\langle u,f\rangle=a$.
Since $v+kf\in Q'$, we obtain that $\langle u+aw,v\rangle =\langle u,v+kf\rangle\geq\ell$. 
If $k<0$, then \eqref{eq:condition_rat} gives $v=g'-kf$ for some $g'\in G'_k$ and $f\in F$. 
Thus, \[\langle u+aw,v\rangle  = \underbrace{\langle u,g'\rangle}_{\geq \ell \text{ by }g' \in Q'}-k(\underbrace{\langle u,f\rangle-a}_{\geq 0}) \geq\ell.\] 

On the other hand, choose $f\in F$ with $\langle u,f\rangle=a$. Since $q\in G_h$ and $h<0$, \eqref{eq:condition_rat} implies that $q-hf\in P$. 
By the above claim, we have $\ell \leq\langle u+aw,q-hf\rangle$, while $\langle u+aw,q-hf\rangle =\langle u,q\rangle  <\ell$, a contradiction. Hence, $Q=Q'$.

Now let $\{G_h\}$ be arbitrary and set $Q=\mut_w(P,F;\{G_h\})$. Let $Q_0$ be obtained by replacing $G_h$ by $\emptyset$ whenever $V(P)\cap H_{w,h}=\emptyset$. 
Then $Q_0\subseteq Q$ is a rational polytope. If $Q\not\subseteq Q_0$, then the preceding argument, with $Q_0$ in place of $Q'$, gives the same contradiction. 
Thus, $Q=Q_0$. Applying the same argument to $\{G'_h\}$, we obtain $Q'=Q'_0$, where $Q'_0$ is defined analogously. 
By the case already proved, $Q_0=Q'_0$. Hence, $Q=Q'$.

Finally, fix $u\in M_\RR$ and put $a=\min_{f\in F}\langle u,f\rangle$. 
For any $\ell\in\RR$, we have \[ Q\subseteq H_{u,\geq\ell} \quad\Longleftrightarrow\quad P\subseteq H_{u+aw,\geq\ell}.\]
The implication $\Rightarrow$ follows from \eqref{eq:PQ}. 
For the reverse implication, for $h\geq0$, $p\in P_{w,h}$ and $f\in F$, we have 
\[ \langle u,p+hf\rangle =\langle u,p\rangle+h\langle u,f\rangle \geq\langle u,p\rangle+ha =\langle u+aw,p\rangle\geq\ell.\]
For $h\leq0$ and $g\in G_h$, choose $f'\in F$ with $\langle u,f'\rangle=a$. Then $p:=g-hf'\in P_{w,h}$ by \eqref{eq:condition_rat}, and
\[ \langle u,g\rangle =\langle u,p+hf'\rangle =\langle u,p\rangle+ha =\langle u+aw,p\rangle\geq\ell.\]
Thus, $Q\subseteq H_{u,\geq\ell}$. 
Taking the largest possible $\ell$ proves \eqref{eq:support-polytope}. 
\end{proof}

The following proposition describes the effect of translating the factor.
In particular, when $F$ is a lattice polytope, a lattice translation allows us to place one of its vertices at the origin. 
\begin{prop}\label{prop:trans}
Let $P \subseteq N_\RR$ be a rational polytope. Let $w \in M$ and let $F \subseteq w^\perp$ be a rational polytope. 
Suppose that $(P,w,F)$ satisfies \eqref{eq:condition_rat} with $\{G_h\}$. 
Then, for any $u \in w^\perp \cap N$, $\mut_w(P,F+u;\{G_h+hu\})$ is unimodularly equivalent to $\mut_w(P,F;\{G_h\})$. 
\end{prop}
\begin{proof}
First, we notice that $(P,w,F+u)$ satisfies \eqref{eq:condition_rat} with the family $\{G_h+hu\}_{h\leq0}$ by \[ (G_h+hu)+(-h)(F+u) = G_h+(-h)F.\]

Define a linear map $\psi:N_\RR\to N_\RR$ by $\psi(v)=v+\langle w,v\rangle u$. 
Since $u\in w^\perp\cap N$, the map $\psi$ preserves the lattice $N$, and its inverse is given by $\psi^{-1}(v)=v-\langle w,v\rangle u$. 
Hence, $\psi$ is unimodular. 

For $v\in H_{w,h}$, we have $\psi(v)=v+hu$. Thus, $\psi(G_h)=G_h+hu$ for $h\leq0$, and $\psi(P_{w,h}+hF)=P_{w,h}+h(F+u)$ for $h\geq0$. Thus, 
\[\begin{aligned} \psi\bigl(\mut_w(P,F;\{G_h\})\bigr) &= \conv\left( \bigcup_{h\leq0}(G_h+hu) \,\cup\, \bigcup_{h\geq0}\bigl(P_{w,h}+h(F+u)\bigr) \right)\\
 &= \mut_w(P,F+u;\{G_h+hu\}),
 \end{aligned}\] as required. 
\end{proof}

The preceding results do not require $w$ to be primitive or $F$ to be a lattice polytope. 
In the remaining sections, we assume that $w\in M$ is primitive and that $F$ is a lattice polytope. 
The resulting piecewise-linear maps preserve the lattice. 

\subsection{Combinatorial mutation for rational polyhedral pointed cones}\label{subsec:cone}

Next, we introduce combinatorial mutation for \textit{rational polyhedral pointed cones}. 
Let $C=\cone(\{v_1,\ldots,v_m\}) \subseteq N_\RR$, where the $v_i\in N$ are the primitive generators of its extremal rays. 
Set $V(C)=\{v_1,\ldots,v_m\}$ and $C_{w,h}=C\cap H_{w,h}$. For $C=\{0\}$, we set $V(C)=\emptyset$. 

\begin{defi}[Combinatorial mutation for rational polyhedral pointed cones]\label{def:cone}
Take a primitive $w\in M$ and a lattice polytope $F\subseteq w^\perp$.
Suppose that, for each $h\leq0$, there is a possibly empty polyhedron
$G_h\subseteq H_{w,h}$, rational whenever $h$ is rational, such that
\begin{equation}\label{eq:condition_cone}
 V(C)\cap H_{w,h}\subseteq G_h+(-h)F\subseteq C_{w,h}.
\end{equation}
Define
\[ \mut_w(C,F;\{G_h\}) :=\cone\left(\bigcup_{h\leq0}G_h \ \cup\ \bigcup_{h\geq0}(C_{w,h}+hF)\right).\]
As explained below, this is a rational polyhedral pointed cone independent of the family; we denote it by $\mut_w(C,F)$. 
\end{defi}

\begin{rem}\label{rem:cone}
We explain how to apply the proofs of Lemmas~\ref{lem:1}, \ref{lem:2} and Proposition~\ref{prop:G_h} to cones. 
Let $C'$ be the cone obtained from $\{G_h\}_{h \leq 0}$ satisfying \eqref{eq:condition_cone}. 
In the proof of Lemma~\ref{lem:1}, keep $\lambda_i\geq0$ but omit the condition $\sum_i\lambda_i=1$.
The argument still works because $\sum_i\lambda_i c_i\in C$ whenever $c_i\in C$. 
It gives \[ C'_{w,h}=C_{w,h}+hF\quad(h\geq0),\qquad  C'_{w,h}+(-h)F\subseteq C_{w,h}\quad(h\leq0).\]
Then these relations and \eqref{eq:condition_cone} give $\mut_{-w}(C',F)=C$, as in Lemma~\ref{lem:1}. 

For Proposition~\ref{prop:G_h}, replace the vertices of $P$ by the elements of $V(C)$, and use $H_{u,\geq0}$ in place of $H_{u,\geq\ell}$. 
More precisely, for $u\in M_\RR$, put $a=\min_{f\in F}\langle u,f\rangle$. Then we see that 
\begin{equation}\label{eq:support-cone}
 C'\subseteq H_{u,\geq0}
 \quad\Longleftrightarrow\quad
 C\subseteq H_{u+aw,\geq0}.
\end{equation}
In both directions, the calculations are identical with those in the proof of Proposition~\ref{prop:G_h} with $\ell=0$. 

Keeping only $G_h$ with $V(C)\cap H_{w,h}\neq\emptyset$ gives a rational polyhedral cone, 
and \eqref{eq:support-cone} shows that adding the other $G_h$ does not change this cone. 
Hence, the combinatorial mutation is independent of the choice of $\{G_h\}$. 

These also imply pointedness. Indeed, let us assume $h:=\langle w,v\rangle\geq0$ when $v,-v\in C'$. 
Then $v-hf,-v+hf\in C$ for some $f\in F$. Since $C$ is pointed, $v=hf$, so $\langle w,v \rangle = h\langle w,f \rangle = 0$, i.e., $v=0$. 

Finally, the proof of Lemma~\ref{lem:2}, with vertices replaced by primitive generators of extremal rays, gives 
\begin{equation}\label{eq:V(C)}
\begin{split}
 V(C')&\subseteq
 \{c+\langle w,c\rangle f:c\in C,\ f\in V(F)\},\\
 V(C)&\subseteq
 \{c'-\langle w,c'\rangle f:c'\in C',\ f\in V(F)\}.
\end{split}
\end{equation}
\end{rem}

\subsection{Combinatorial mutation for $\calP_N$}\label{subsec:PN}
We now extend the construction to $P\in\calP_N$ using $P_{w,h}:=P\cap H_{w,h}$ and the combinatorial mutation of $\rec(P)$. 

\begin{defi}[Combinatorial mutation for $P\in\calP_N$]\label{def:polyhedron}
Let $P\in\calP_N$, set $C=\rec(P)$, and take a primitive $w\in M$ and a lattice polytope $F\subseteq w^\perp$. 
Suppose that $(C,w,F)$ satisfies \eqref{eq:condition_cone} and that, for every $h\leq0$, 
there is a possibly empty polyhedron $G_h\subseteq H_{w,h}$, rational whenever $h$ is rational, such that
\begin{equation}\label{eq:condition_polyhedron}
 V(P)\cap H_{w,h}\subseteq G_h+(-h)F\subseteq P_{w,h}.
\end{equation}
In this case, we say that $\mut_w(P,F)$ is \textit{well-defined} and set
\begin{equation}\label{eq:mutation_polyhedron}
 \mut_w(P,F) :=\conv\left(\bigcup_{h\leq0}G_h \ \cup \ \bigcup_{h\geq0}(P_{w,h}+hF)\right)+\mut_w(C,F).
\end{equation}
For bounded $P$, this agrees with Definition~\ref{def:mutation}. 
\end{defi}

It is enough to retain $G_h$ only for the finitely many $h\leq0$ such that $V(P)\cap H_{w,h}\neq\emptyset$. 
Together with $\mut_w(C,F)$, these sets define a rational polyhedron whose recession cone is $\mut_w(C,F)$. 
The same argument as in Proposition~\ref{prop:G_h} shows that adding the other $G_h$ does not change the resulting polyhedron. 
In particular, the definition is independent of the choice of $\{G_h\}$. 

Moreover, for every $u\in M_\RR$, we have
\begin{equation}\label{eq:support-polyhedron}
 \inf_{q\in\mut_w(P,F)}\langle u,q\rangle
 =
 \inf_{p\in P}\langle u+aw,p\rangle,
\end{equation}
where $a=\min_{f\in F}\langle u,f\rangle$. Indeed, if the left-hand side is $-\infty$, so is the right-hand side
by \eqref{eq:support-cone}; otherwise the same calculation as in
Proposition~\ref{prop:G_h} applies to a vertex of $P$.

Finally, since $0\in\operatorname{int}(P)$ and
$u+aw\neq0$ for every $u\neq0$, \eqref{eq:support-polyhedron}
implies that $0\in\operatorname{int}(\mut_w(P,F))$.
Its recession cone $\mut_w(C,F)$ is pointed, and hence
$\mut_w(P,F)\in\calP_N$.

If $P=P'+C$, where $P'$ is a rational polytope, and both $(P',w,F)$ and $(C,w,F)$ admit combinatorial mutations, 
then $\mut_w(P,F)$ is well-defined and 
\begin{equation}\label{eq:sum_mutation}
 \mut_w(P,F)=\mut_w(P',F)+\mut_w(C,F).
\end{equation}
In fact, $P'\subseteq P$ and $V(P)\subseteq V(P')$, so the family used for $P'$ satisfies \eqref{eq:condition_polyhedron}. 
The equality follows from \eqref{eq:support-polytope}--\eqref{eq:support-polyhedron}, 
since the infimum over a Minkowski sum is the sum of the two infima.

\begin{ex}[Continuation to Example~\ref{ex:running1}]\label{ex:running2}
Let $P,P',C$ be the same as in Example~\ref{ex:running1}. Note that the decomposition $P=P'+C$ is the one in \eqref{eq:decomp}. 
Let $w=(1,1)$ and $F=\conv(\{(0,0),(1,-1)\})$. Then $F \subseteq w^\perp$. 
Moreover, by letting $G_{-1}$ be a $0$-dimensional lattice polytope $\{(-1,0)\}$ and $G_h=\emptyset$ for $h \leq 0$ with $h \neq -1$, 
the condition \eqref{eq:condition_rat} for $P'$ is satisfied. 
The condition \eqref{eq:condition_cone} for $C$ automatically holds since $C_{w,h} =\emptyset$ for $h < 0$. 
(We may set $G_0=\{(0,0)\}$.) 
We may therefore use \eqref{eq:sum_mutation} and compute the mutations of the two summands separately: 
\begin{align*}
&\mut_w(P',F)=G_{-1} \text{ and }\mut_w(C,F)=\cone(C_{w,1}+F), \text{ so }\\
&\mut_w(P,F)=(-1,0)+\cone(\{(2,-1),(0,1)\}). 
\end{align*}
See Figure~\ref{fig1}. 
\end{ex}
\begin{figure}[htb!]
\centering
\begin{tikzpicture}[scale=1.5]
\begin{scope}[xshift=-0.5cm]

\begin{scope}
  \clip (-1.5,-1.5) rectangle (1.8,1.8);

  \fill[blue!20,opacity=0.7]
    (-1,-1) -- (-1,1.8) -- (1.8,1.8) -- (1.8,-1) -- cycle;

  \clip (-1,0) -- (0,-1) -- (1.8,-1) -- (1.8,1.8) -- (-1,1.8) -- cycle;
\end{scope}

\fill[white]
  (-1,-1) -- (-1,0) -- (0,-1) -- cycle;

\draw[->] (-1.5,0) -- (1.6,0);
\draw[->] (0,-1.5) -- (0,1.6);

\draw[thick] (-1,0) -- (-1,1.6);
\draw[thick] (0,-1) -- (1.6,-1);
\draw[thick] (-1,0) -- (0,-1);

\draw[ultra thick] (0,0) -- (1,-1);
\node[above right] at (0.55,-0.45) {$F$};
\draw[->,very thick] (0,0)--(.7,.7) node[above left] {$w$};

\fill (-1,0) circle (0.03);
\fill (0,-1) circle (0.03);

\node[left]  at (-1,0.15) {$(-1,0)$};
\node[below] at (0.37,-1) {$(0,-1)$};

\node at (0.6,1.55) {{\Large$P$}};

\end{scope}

\begin{scope}[xshift=4cm]

\begin{scope}
  \clip (-1.5,-1.5) rectangle (1.8,1.8);

  \clip (-1,-1.5) -- (-1,1.8) -- (1.8,1.8) -- (1.8,-1.5) -- cycle;

  \clip (-1,0) -- (3,-2) -- (1.8,0.4) -- (1.8,1.8) -- (-1,1.8) -- cycle;

  \fill[blue!20,opacity=0.7]
    (-1.5,-1.5) rectangle (1.8,1.8);
\end{scope}

\draw[->] (-1.5,0) -- (1.6,0);
\draw[->] (0,-1.5) -- (0,1.6);

\draw[thick] (-1,0) -- (-1,1.6);
\draw[thick] (-1,0) -- (1.6,-1.3);

\draw[ultra thick] (0,0) -- (1,-1);
\node[above right] at (0.55,-0.45) {$F$};
\draw[->,very thick] (0,0)--(.7,.7) node[above left] {$w$};

\fill (-1,0) circle (0.03);
\node[left] at (-1,0.15) {$(-1,0)$};

\node at (1,1.55) {{\Large$\mut_w(P,F)$}};

\end{scope}
\end{tikzpicture}
\caption{Combinatorial mutation $\mut_w(P,F)$ for a polyhedron in $\calP_N$}\label{fig1}
\end{figure}

\bigskip

\FloatBarrier
\section{Dual operation of combinatorial mutation}\label{sec:M}

We have introduced combinatorial mutation for $P\in\calP_N$ in $N_\RR$.
In this section, we analyze the corresponding operation on
$P^*\in\calQ_M$.

The map $\varphi_{w,F}$ below is the same as the map in
\cite[Section~3]{ACGK}. Proposition~\ref{prop:commutative} extends
its relation with combinatorial mutation to unbounded polyhedra,
using \eqref{eq:support-polyhedron}.

\begin{defi}[{cf. \cite[Section~3]{ACGK}}]\label{def:tropical}
Fix a primitive lattice vector $w\in M$ and a lattice polytope
$F\subseteq w^\perp$. We define $\varphi_{w,F}:M_\RR\to M_\RR$ by
\[
 \varphi_{w,F}(u):=u-u_{\min}w,
 \qquad u_{\min}:=\min_{v\in F}\langle u,v\rangle.
\]
Since $F\subseteq w^\perp$, we have
$\varphi_{-w,F}\circ\varphi_{w,F}=\operatorname{id}$.
In particular, $\varphi_{w,F}^{-1}=\varphi_{-w,F}$.
\end{defi}

As in \cite[Section~3]{ACGK}, let $\Delta_F$ be the normal fan of $F$,
with normal cones defined by maximization of linear functionals.
Every cone of $\Delta_F$ contains $\RR w$, so none is pointed.
Let $\calF(\Delta_F)$ be the set of maximal cones, and let
$v_\sigma\in V(F)$ correspond to $\sigma\in\calF(\Delta_F)$.
Define
\[
 \varphi_{w,\sigma}(u):=u-\langle u,v_\sigma\rangle w.
\]
Its inverse is $u\mapsto u+\langle u,v_\sigma\rangle w$, so
$\varphi_{w,\sigma}\in\GL(M)$.
For $u\in-\sigma$, we have
\[
 \varphi_{w,F}(u)=\varphi_{w,\sigma}(u)
 =u-\langle u,v_\sigma\rangle w.
\]
If $u\in(-\sigma)\cap(-\tau)$, then
$\langle u,v_\sigma\rangle=\langle u,v_\tau\rangle=u_{\min}$,
and hence $\varphi_{w,\sigma}(u)=\varphi_{w,\tau}(u)$.

\begin{ex}\label{ex:tropical}
Let $w=(1,1)$ and $F=\conv(\{(0,0),(1,-1)\})$. For $(x,y)\in\RR^2$, we have 
\[
 \varphi_{w,F}((x,y))
 =(x,y)-\min\{0,x-y\}(1,1)
 =\begin{cases}
 (x,y),&x\geq y,\\
 (y,-x+2y),&x\leq y.
 \end{cases}
\]
Figure~\ref{fig:tropical} shows $\Delta_F$.
Let $\sigma,\tau\in\calF(\Delta_F)$ correspond to $(0,0)$ and
$(1,-1)$, respectively. Using row vectors, the matrices of
$\varphi_{w,\sigma}$ and $\varphi_{w,\tau}$ are
\[
 \begin{pmatrix}1&0\\0&1\end{pmatrix}
 \qquad\text{and}\qquad
 \begin{pmatrix}0&-1\\1&2\end{pmatrix},
\]
respectively. These give $\varphi_{w,F}$ on $-\sigma$ and $-\tau$.

\begin{figure}[htbp]
\centering
\begin{tikzpicture}[scale=0.5]

\path[use as bounding box] (-3.4,-3.4) rectangle (3.8,3.8);

\begin{scope}
  \clip (-3.2,-3.2) -- (3.2,3.2) -- (3.2,-3.2) -- cycle;
  \fill[blue!20] (-3.2,-3.2) rectangle (3.2,3.2);
\end{scope}

\begin{scope}
  \clip (-3.2,-3.2) -- (3.2,3.2) -- (-3.2,3.2) -- cycle;
  \fill[red!20] (-3.2,-3.2) rectangle (3.2,3.2);
\end{scope}

\draw[->] (-3.2,0) -- (3.3,0) node[right] {$x$};
\draw[->] (0,-3.2) -- (0,3.3) node[above] {$y$};

\draw[thick] (-3,-3) -- (3,3);
\end{tikzpicture}
\caption{The normal fan $\Delta_F$: the red region corresponds
to $(0,0)$ and the blue region to $(1,-1)$.}
\label{fig:tropical}
\end{figure}
\end{ex}

For a rational polytope $Q\subseteq M_\RR$, the number
$\#(nQ\cap M)$, as a function of $n\in\ZZ_{>0}$, is a
quasi-polynomial, called the \textit{Ehrhart quasi-polynomial} of $Q$.
See \cite{BR} for an introduction to Ehrhart theory.
For every $m\in\ZZ_{>0}$, we have
$\varphi_{w,F}(mQ)=m\varphi_{w,F}(Q)$.
Since each $\varphi_{w,\sigma}$ is unimodular and
$\varphi_{-w,F}$ is the inverse of $\varphi_{w,F}$, the latter induces
a bijection between $mQ\cap M$ and $m\varphi_{w,F}(Q)\cap M$. 
Hence, if $\varphi_{w,F}(Q)$ is convex, then it is a rational polytope with the same Ehrhart quasi-polynomial as $Q$ (cf.~\cite[Proposition~4]{ACGK}). 
In particular, all rational polytopes in $M_\RR$ which are combinatorially mutation-equivalent to $Q$ have the same Ehrhart quasi-polynomial. 
For combinatorial mutations in $N_\RR$, the corresponding statement concerns the polar duals, as the following proposition shows. 
Figure~\ref{fig2} illustrates this relation.
\begin{prop}[{cf. \cite{ACGK}, \cite[Proposition~6.6]{HN}}]
\label{prop:commutative}
Let $P\in\calP_N$, let $w\in M$ be primitive, and let
$F\subseteq w^\perp$ be a lattice polytope.
If $\mut_w(P,F)$ is well-defined, then
\[ \varphi_{w,F}(P^*)=\mut_w(P,F)^*.\]
Thus the following diagram commutes whenever the combinatorial mutations are defined: 
\[\begin{tikzcd}
\calP_N
  \arrow[r, "(\cdot)^*"]
  \arrow[d, "{\mut_w(\cdot,F)}" swap]
&
\calQ_M
  \arrow[d, "\varphi_{w,F}(\cdot)"]
\\
\calP_N
  \arrow[r, "(\cdot)^*"]
&
\calQ_M
  \arrow[ul, phantom, "\circlearrowright" description, yshift=0.2em]
\end{tikzcd}\]
\end{prop}
\begin{proof}
Write $R=\mut_w(P,F)$. Take $u\in P^*$ and put $a=\min_{f\in F}\langle u,f\rangle$.
Since $F\subseteq w^\perp$, we have $\min_{f \in F}\langle u-aw, f \rangle=a$. 
By \eqref{eq:support-polyhedron}, we see that 
\[ \inf_{q\in R}\langle u-aw,q\rangle =\inf_{p\in P}\langle u,p\rangle\geq-1.\]
Thus, $\varphi_{w,F}(u)=u-aw\in R^*$.

Conversely, take $u\in R^*$ and choose $\sigma\in\calF(\Delta_F)$ with $u\in-\sigma$. 
Then $\langle u,v_\sigma\rangle=\min_{f\in F}\langle u,f\rangle$. 
Set $u':=u+\langle u,v_\sigma\rangle w$.
Then \eqref{eq:support-polyhedron} implies that  
$\inf_{p\in P}\langle u',p\rangle =\inf_{q\in R}\langle u,q\rangle\geq-1$, so $u'\in P^*$. 
Since $\langle u',f\rangle=\langle u,f\rangle$ for any $f\in F$, we have $\varphi_{w,F}(u')=u$, i.e., $u \in \varphi_{w,F}(P^*)$, as required. 
\end{proof}

\begin{ex}\label{ex:running3}
Recall that $w=(1,1)$ and $F=\conv(\{(0,0),(1,-1)\})$, and let $P$ be as in Example~\ref{ex:running2}. We have 
\[ \mut_w(P,F)=\{(-1,0)\}+\cone(\{(2,-1),(0,1)\}),\]
and hence
\[ \mut_w(P,F)^*=H_{(-1,0),\geq-1}\cap H_{(2,-1),\geq0}\cap H_{(0,1),\geq0} =\conv(\{(0,0),(1,0),(1,2)\}).\]
On the other hand, we know that $P^*=\conv(\{(0,0),(1,0),(0,1),(1,1)\})$. The maximal cones of $\Delta_F$ are 
\[ \sigma_1=\cone(\{\pm(1,1),(-1,1)\})\;\;\text{ and }\;\; \sigma_2=\cone(\{\pm(1,1),(1,-1)\}),\]
corresponding to $(0,0)$ and $(1,-1)$, respectively.
Thus $\varphi_{w,F}(u)=u$ on $-\sigma_1$, whereas 
\[ \varphi_{w,F}(u) =u-\langle u,(1,-1)\rangle w =u\begin{pmatrix}0&-1\\1&2\end{pmatrix}\]
on $-\sigma_2$. 
Applying these maps to the two triangles
$P^*\cap(-\sigma_1)$ and $P^*\cap(-\sigma_2)$, we obtain \[ \varphi_{w,F}(P^*)=\conv(\{(0,0),(1,0),(1,2)\}).\] 
In particular, $\varphi_{w,F}((0,1))=(1,2)$. See Figure~\ref{fig2}.
\begin{figure}[htbp]
\centering
\begin{tikzpicture}[scale=1.1]
\begin{scope}

\clip (-0.9,-0.9) rectangle (3.9,3.9);

\begin{scope}
  \clip (-3.2,-3.2) -- (3.2,3.2) -- (3.2,-3.2) -- cycle;
  \fill[blue!20] (-3.2,-3.2) rectangle (3.2,3.2);
\end{scope}

\begin{scope}
  \clip (-3.2,-3.2) -- (3.2,3.2) -- (-3.2,3.2) -- cycle;
  \fill[red!20] (-3.2,-3.2) rectangle (3.2,3.2);
\end{scope}

\draw[->] (-0.9,0) -- (3,0);
\draw[->] (0,-0.9) -- (0,3);

\draw[thick] (-3,-3) -- (3,3);

\draw[very thick]
  (0,0) -- (1,0) -- (1,1) -- (0,1) -- cycle;

\fill (0,0) circle (0.07);
\fill (1,0) circle (0.07);
\fill (0,1) circle (0.07);
\fill (1,1) circle (0.07);

\node[below left, font=\scriptsize] at (0.05,0.5) {$(0,0)$};
\node[right,      font=\scriptsize] at (0.9,0.25) {$(1,0)$};
\node[left,       font=\scriptsize] at (0.05,1) {$(0,1)$};
\node[right,      font=\scriptsize] at (0.9,0.85) {$(1,1)$};

\node[font=\large] at (0.5,1.4) {$P^*$};

\end{scope}

\begin{scope}[xshift=5.2cm]

\clip (-0.9,-0.9) rectangle (3.9,3.9);

\begin{scope}
  \clip (-3.2,-3.2) -- (3.2,3.2) -- (3.2,-3.2) -- cycle;
  \fill[blue!20] (-3.2,-3.2) rectangle (3.2,3.2);
\end{scope}

\begin{scope}
  \clip (-3.2,-3.2) -- (3.2,3.2) -- (-3.2,3.2) -- cycle;
  \fill[red!20] (-3.2,-3.2) rectangle (3.2,3.2);
\end{scope}

\draw[->] (-0.9,0) -- (3,0);
\draw[->] (0,-0.9) -- (0,3);

\draw[thick] (-3,-3) -- (3,3);

\draw[very thick]
  (0,0) -- (1,0) -- (1,1) -- (1,2) -- cycle;

\fill (1,1) circle (0.07);

\fill (0,0) circle (0.07);
\fill (1,0) circle (0.07);
\fill (1,2) circle (0.07);

\node[below left, font=\scriptsize] at (0.1,0.6) {$(0,0)$};
\node[right,      font=\scriptsize] at (1,0.25) {$(1,0)$};
\node[right,      font=\scriptsize] at (1,1) {$(1,1)$};
\node[right,      font=\scriptsize] at (1,2) {$(1,2)$};

\node[font=\small] at (1.55,3.5)
  {$\varphi_{w,F}(P^*)=\mut_w(P,F)^*$};

\end{scope}
\end{tikzpicture}

\caption{Example of a combinatorial mutation $\varphi_{w,F}$
for a rational polytope $P^*$ in $\calQ_M$}
\label{fig2}
\end{figure}
\end{ex}

The image $\varphi_{w,F}(Q)$ need not be convex, even when $Q$ is.
For example, with $w,F$ as in Example~\ref{ex:tropical}, take $Q=\conv(\{(0,0),(1,0),(0,1)\})$. 
Its image contains $(1,0)$ and $(1,2)$ but not their midpoint $(1,1)$ since $\varphi_{-w,F}((1,1))=(1,1)\notin Q$.

The following proposition characterizes when the image is convex.

\begin{prop}\label{prop:convexity}
Let $w\in M$ be primitive, let $F\subseteq w^\perp$ be a lattice polytope and let $Q\in\calQ_M$.
Then $\mut_w(Q^*,F)$ is well-defined if and only if $\varphi_{w,F}(Q)$ is convex.
\end{prop}
\begin{proof}
The ``only if'' direction follows from Proposition~\ref{prop:commutative}. 

Conversely, suppose that $R:=\varphi_{w,F}(Q)$ is convex, and put $P=Q^*$ and $C=\rec(P)$. 
We verify \eqref{eq:condition_polyhedron} and \eqref{eq:condition_cone}. 

First, let $v\in V(P)$ and suppose that $h:=\langle w,v\rangle<0$. 
We will construct a point $g\in H_{w,h}$ such that 
\begin{equation}\label{eq:g-polytope}
 v\in g+(-h)F\subseteq P_{w,h}.
\end{equation}
Since $P=Q^*$, the vertex $v$ corresponds to the facet $F_v:=\{x\in Q:\langle x,v\rangle=-1\}$ of $Q$. 
Choose a point $x_0$ in the relative interior of $F_v$ at which there is a unique $f_0 \in F$ with $\min_{f \in F}\langle x_0,f \rangle=\langle x_0,f_0 \rangle$. 
There exists such a point. In fact, for distinct $f,f'\in V(F)$, the equality $\langle x,f\rangle=\langle x,f'\rangle$ defines the hyperplane $(f-f')^\perp$.
None of these finitely many hyperplanes contains $F_v$: the normal vector $v$ satisfies $\langle w,v\rangle=h<0$, whereas $\langle w,f-f'\rangle=0$. 
Thus, for some $\sigma\in\calF(\Delta_F)$, we have $x_0\in\operatorname{int}(-\sigma)$, and $\varphi_{w,F}$ agrees with $\varphi_{w,\sigma}$ near $x_0$. 

Let $v_\sigma$ be the vertex of $F$ corresponding to $\sigma$ and set $g:=v+h v_\sigma$. 
If $y=\varphi_{w,\sigma}(x)$, then $\langle x,v\rangle =\langle y,g\rangle$. 
Hence, near $\varphi_{w,F}(x_0)$, the image $R$ satisfies $\langle y,g\rangle\geq-1$ with equality on the image of $F_v$. 
Since $R$ is convex, this inequality holds for every $y\in R$; 
otherwise the segment joining a point violating the inequality to $\varphi_{w,F}(x_0)$ would violate it arbitrarily close to $\varphi_{w,F}(x_0)$. 

Now, take $x\in Q$ and $f\in F$, and put $a:=\min_{f'\in F}\langle x,f'\rangle$. 
Since $\langle w,g\rangle=h$, we have 
\[ \langle x,g-hf\rangle = \langle\varphi_{w,F}(x),g\rangle +h\bigl(a-\langle x,f\rangle\bigr) \geq -1\]
because $h<0$ and $a\leq\langle x,f\rangle$. Thus, $g-hf\in Q^*=P$ for every $f\in F$. Moreover, we also have $\langle w,g-hf\rangle=h$, 
so $g+(-h)F\subseteq P_{w,h}$ holds. Finally, we obtain that \[ v=g-hv_\sigma\in g+(-h)F,\] which proves \eqref{eq:g-polytope}. 

For each $h<0$, let $G_h$ be the convex hull of the points $g$ constructed above for all $v\in V(P)\cap H_{w,h}$, 
and set $G_h=\emptyset$ if there is no such vertex. At height zero, set $G_0=P_{w,0}$. Then
\[ V(P)\cap H_{w,h} \subseteq G_h+(-h)F \subseteq P_{w,h}\]for every $h\leq0$.

For $c\in V(C)$ with $h:=\langle w,c\rangle<0$, we apply the same construction to the facet $\{x\in Q \mid \langle x,c\rangle=0\}$. 
It gives a point $g\in H_{w,h}$ such that $c\in g+(-h)F$ and $\langle x,g-hf\rangle\geq0$ for all $x\in Q$, $f\in F$. 
Since \[ C=\rec(P) =\{c'\in N_\RR \mid \langle x,c'\rangle\geq0 \text{ for all }x\in Q\},\]
we obtain\[ c\in g+(-h)F\subseteq C_{w,h}.\]
Taking the convex hull of these points $g$ at each negative height and taking $C_{w,0}$ at height zero gives \eqref{eq:condition_cone}.

All the points $g$ constructed above are rational, and there are only finitely many relevant heights. 
Hence, \eqref{eq:condition_polyhedron} and \eqref{eq:condition_cone} are satisfied, and therefore $\mut_w(Q^*,F)$ is well-defined. 
\end{proof}

By Propositions~\ref{prop:commutative} and~\ref{prop:convexity}, we obtain that 
\[ \mut_{-w}(\mut_w(P,F),F)=P \]
whenever $\mut_w(P,F)$ is defined.
The same propositions show that $P_1,P_2\in\calP_N$ are combinatorially mutation-equivalent in $N_\RR$ 
if and only if $P_1^*,P_2^*$ are combinatorially mutation-equivalent in $M_\RR$. 
Equivalently, $Q_1,Q_2\in\calQ_M$ are combinatorially mutation-equivalent in $M_\RR$ 
if and only if $Q_1^*,Q_2^*$ are combinatorially mutation-equivalent in $N_\RR$.

\bigskip

\FloatBarrier
\section{Combinatorial mutation-equivalence of two poset polytopes}\label{sec:poset}

The goal of this section is to prove the combinatorial mutation-equivalence of two poset polytopes (order polytopes and chain polytopes). 

We recall the basic definitions of order and chain polytopes, introduced in \cite{S86}. 
Let $\Pi$ be a nonempty finite poset equipped with a partial order $\preceq$. 
\begin{itemize}
\item For $p,q \in \Pi$, we say that $p$ \emph{covers} $q$ if $q \prec p$ and there is no $p'$ with $p' \neq p,p'\neq q$ and $q \prec p' \prec p$. 
\item We say that $C \subseteq \Pi$ is a \emph{poset filter} if $C$ satisfies the following condition: 
$$p \in C \text{ and }p \preceq q \;\Longrightarrow q \in C. $$
Note that an empty set is regarded as a poset filter. 
\item We say that $A \subseteq \Pi$ is an \emph{antichain} if $p \not\preceq q$ and $q \not\preceq p$ hold for any $p,q \in A$ with $p \neq q$. 
Note that an empty set is regarded as an antichain. 
\end{itemize}

We define two poset polytopes $\calO(\Pi)$ and $\calC(\Pi)$ as follows: 
\begin{align*}
\calO(\Pi)&=\{ (x_p)_{p \in \Pi} \in \RR^\Pi \mid x_p \leq x_q \text{ if }p \preceq q \text{ in }\Pi, \; 0 \leq x_p \leq 1 \text{ for }\forall p \in \Pi\}, \\
\calC(\Pi)&=\{ (x_p)_{p \in \Pi} \in \RR^\Pi \mid
x_{p_{i_1}}+\cdots+x_{p_{i_\ell}} \leq 1 \text{ if }p_{i_1} \prec \cdots \prec p_{i_\ell} \text{ in }\Pi, \; x_p \geq 0 \text{ for }\forall p \in \Pi\}, 
\end{align*}
where $\RR^\Pi$ stands for $\RR^d$ with $d=|\Pi|$ whose entry is indexed by the elements of $\Pi$. 
The convex polytope $\calO(\Pi)$ is called the \emph{order polytope} of $\Pi$, and $\calC(\Pi)$ is called the \emph{chain polytope} of $\Pi$. 
It is known that both $\calO(\Pi)$ and $\calC(\Pi)$ are $(0,1)$-polytopes, i.e., all vertices are $(0,1)$-vectors. 
The vertices of $\calO(\Pi)$ (resp. $\calC(\Pi)$) are in bijection with the poset filters (resp. the antichains) of $\Pi$. 
More precisely, a $(0,1)$-vector $(x_p)_{p \in \Pi}$ is a vertex of $\calO(\Pi)$ (resp. $\calC(\Pi)$) 
if and only if $\{p \mid x_p =1\}$ is a poset filter (resp. an antichain) of $\Pi$. 
Hence, both $\calO(\Pi)$ and $\calC(\Pi)$ contain the origin on their boundaries. 

The poset filters are in bijection with the antichains: we may take minimal elements of each poset filter to get an antichain; 
we may take the elements greater than or equal to the elements of each antichain to get a poset filter. 
Thus, the number of vertices of $\calO(\Pi)$ is equal to that of $\calC(\Pi)$. 
Since those are $(0,1)$-polytopes, we conclude that $\# (\calO(\Pi) \cap \ZZ^\Pi)=\# (\calC(\Pi) \cap \ZZ^\Pi)$. 
Note that $\calO(\Pi)$ and $\calC(\Pi)$ need not be unimodularly equivalent. 
The necessary and sufficient condition for $\calO(\Pi)$ and $\calC(\Pi)$ to be unimodularly equivalent is provided in \cite[Theorem 2.1]{HL}. 
Moreover, it is proved in \cite{S86} that $\# (m\calO(\Pi) \cap \ZZ^\Pi)=\# (m\calC(\Pi) \cap \ZZ^\Pi)$ holds for any positive integer $m$, 
so the two polytopes have the same Ehrhart polynomial. 
This is proved by using the transfer map $\phi:\RR^\Pi \to \RR^{\Pi}$ defined by $\phi((x_p)_{p\in\Pi})= (\phi(x)_p)_{p\in\Pi}$, where
\[ \phi(x)_p = \begin{cases}
 x_p, & \text{if $p$ is minimal in $\Pi$},\\[1mm]
 \displaystyle
 \min\{x_p-x_{p'}\mid p\text{ covers }p'\}, & \text{otherwise}.
 \end{cases}\]
Stanley proved that $\phi$ restricts to a bijection from $\calO(\Pi)$ to $\calC(\Pi)$ (\cite[Theorem 3.2]{S86}). 

Next, we recall the chain-order polytopes used below; these are special cases of the marked chain-order polytopes of \cite{FF,FFLP}. 
Let $\widehat{\Pi}=\Pi\sqcup\{\widehat{0},\widehat{1}\}$ be the poset obtained from $\Pi$ by adjoining a new minimum $\widehat{0}$ and a new maximum $\widehat{1}$. 
Let $C\subseteq\Pi$ be a poset filter and set $O=\Pi\setminus C$. 
For $a\in O\sqcup\{\widehat{0},\widehat{1}\}$, put 
\[ y_a= \begin{cases}  0,&a=\widehat{0},\\  1,&a=\widehat{1},\\  x_a,&a\in O. \end{cases}\]
We define the \emph{chain-order polytope} $\calO_C(\Pi)$ by
\[\calO_C(\Pi)=\left\{(x_p)_{p\in\Pi}\in\RR^{\Pi}\ \middle|\
\begin{array}{l}x_p\ge0\quad\text{for all }p\in C,\\[1mm]
x_{p_1}+\cdots+x_{p_r}\le y_b-y_a
\end{array}
\right\},\]
where the second condition is imposed for every saturated chain $a\prec p_1\prec\cdots\prec p_r\prec b$ in $\widehat{\Pi}$ such that
$p_1,\ldots,p_r\in C$ and $a,b\in O\sqcup\{\widehat{0},\widehat{1}\}$ with $r\geq0$; when $r=0$, the inequality is understood as $y_a\leq y_b$. In particular,
\[ \calO_{\emptyset}(\Pi)=\calO(\Pi) \qquad\text{and}\qquad \calO_{\Pi}(\Pi)=\calC(\Pi).\]
\begin{ex}\label{ex:oc}
Let $\Pi=\{p_1,p_2,p_3,p_4\}$ be the poset depicted as in Figure~\ref{fig:poset}. 
\begin{figure}[htb]
\centering
\begin{tikzpicture}[scale=0.8, every node/.style={circle, draw, minimum size=3mm}]
\node (p1) at (0,0) {$p_1$};
\node (p2) at (2,0) {$p_2$};

\node (p3) at (0,2) {$p_3$};
\node (p4) at (2,2) {$p_4$};

\draw (p1) -- (p3);
\draw (p1) -- (p4);
\draw (p2) -- (p4);
\end{tikzpicture}
\caption{Hasse diagram for $\Pi=\{p_1,p_2,p_3,p_4\}$}\label{fig:poset}
\end{figure}
Then
\begin{align*}
\calO(\Pi)&=\conv(\{0,e_3,e_4,e_3+e_4,e_2+e_4,e_1+e_3+e_4,e_2+e_3+e_4,e_1+e_2+e_3+e_4\}), \text{ and}\\
\calC(\Pi)&=\conv(\{0,e_1,e_2,e_3,e_4,e_1+e_2,e_2+e_3,e_3+e_4\}),
\end{align*}
where $e_i\in\RR^4$ denotes the $i$th unit vector. 
Let $C=\{p_4\}$.  Then $C$ is a poset filter of $\Pi$. The corresponding chain-order polytope is 
\[\calO_C(\Pi)=\left\{(x_1,x_2,x_3,x_4)\in\RR^4 \ \middle|\ 
\begin{array}{l}
0\le x_1\le x_3\le1,\\
0\le x_2,\quad 0\le x_4,\\
x_1+x_4\le1,\quad x_2+x_4\le1
\end{array}
\right\}.
\]
Equivalently,
\[\calO_C(\Pi)=\conv(\{0,e_2,e_3,e_4,e_3+e_4,e_1+e_3,e_2+e_3,e_1+e_2+e_3\}).\]
In Example~\ref{ex:int} below, we will see that this polytope occurs naturally as the first intermediate polytope 
in a sequence of combinatorial mutations from $\calO(\Pi)$ to $\calC(\Pi)$.
\end{ex}

It is proved in \cite[Theorem 2.1]{FFLP} that there is a piecewise-linear bijection $\phi_C$ from $\calO(\Pi)$ to $\calO_C(\Pi)$ defined by 
\[ \phi_C(x)_p = \begin{cases}
 x_p, & p\notin C, \text{ or }p\in C\text{ and $p$ is minimal},\\
 \displaystyle
 \min\{x_p-x_{p'}\mid p\text{ covers }p'\}, &\text{otherwise}. 
 \end{cases}\]
In what follows, we regard $\calO(\Pi)$ and $\calC(\Pi)$ as rational polytopes in $M_\RR$. 
In particular, $\calO(\Pi) \in \calQ_M$ and $\calC(\Pi) \in \calQ_M$. The following is the main theorem of this paper. 
\begin{thm}\label{main}
The transfer map $\phi$ can be described as the composition of $\varphi_{w,F}$'s for some $w$'s and $F$'s. 
In particular, $\calO(\Pi)$ and $\calC(\Pi)$ are combinatorially mutation-equivalent in $M_\RR$. 
\end{thm}
\begin{proof}
For each $p \in \Pi$ which is not minimal in $\Pi$, 
let $w_p:=-\eb_p$, $F_p=\conv(\{-\eb_{p'} \mid p \text{ covers }p'\})$ and let $\varphi_p:=\varphi_{w_p,F_p}$, 
where $\eb_p$ for $p \in \Pi$ denotes the $p$-th unit vector of $\RR^\Pi$. Note that $F_p \subseteq w_p^\perp$. 
Our proof consists of three steps: 
\begin{itemize}
\item[(i)] For each non-minimal element $q\in\Pi$, we have 
\[\varphi_q((x_p)_{p \in \Pi})=(x_p')_{p \in \Pi}, \text{ where }x_p'=\begin{cases} \min\{x_p-x_{p'} \mid \text{$p$ covers $p'$}\} &\text{if $p=q$}, \\ x_p &\text{otherwise}. \end{cases}\]
\item[(ii)] Let $q_1,\ldots,q_s \in \Pi$ be all the non-minimal elements of $\Pi$ ordered so that $q_i\prec q_j$ implies $i>j$. Then 
\begin{align}\label{eq:composition}\phi=\varphi_{q_s}\circ \cdots\circ \varphi_{q_2}\circ\varphi_{q_1}.\end{align}
The composition \eqref{eq:composition} is ordered ``from top to bottom''; an empty composition means the identity. 
\item[(iii)] At each intermediate step of \eqref{eq:composition}, the image of $\calO(\Pi)$ is always convex. 
\end{itemize}

\smallskip

First, we prove (i). By definition of $\varphi_q$ (see Definition~\ref{def:tropical}), we see that 
\begin{align*}
\varphi_q((x_p)_{p \in \Pi})=(x_p)_{p \in \Pi}-\left(\min_{v \in F_q}\langle (x_p)_{p \in \Pi}, v \rangle\right)w_q 
=(x_p)_{p \in \Pi}+\left(\min_{\text{$q$ covers $p'$}}(-x_{p'})\right)\eb_q. 
\end{align*}
This proves (i).

Next, the statement of (ii) can be verified from the order of the composition and 
the observation that $\varphi_q$ changes only the $q$-th entry based on $p'$-th entry with $p' \in \Pi$ covered by $q$. 

Finally, let us show (iii).
Let $C_i=\{q_1,\ldots,q_i\}$ for $i=1,\ldots,s$.
Each $C_i$ is a poset filter by the chosen ordering, and
$\varphi_{q_i}\circ\cdots\circ\varphi_{q_1}=\phi_{C_i}$ on $\calO(\Pi)$. 
Then $\phi_{C_i}(\calO(\Pi))=\calO_{C_i}(\Pi)$ by \cite[Theorem 2.1]{FFLP},
and $\calO_{C_i}(\Pi)$ is a lattice polytope by \cite[Proposition 2.4]{FFLP}. 
\end{proof}

Theorem~\ref{main} gives a refinement of the classical relation between the
order polytope and the chain polytope.  Stanley's transfer map already
implies that $\calO(\Pi)$ and $\calC(\Pi)$ have the same
Ehrhart polynomial.  Our theorem shows, more precisely, that this
transfer map can be decomposed into elementary piecewise-unimodular
maps arising from combinatorial mutations.  Thus, although
$\calO(\Pi)$ and $\calC(\Pi)$ are not unimodularly
equivalent in general, they belong to the same combinatorial
mutation-equivalence class.

\begin{ex}\label{ex:int}
Let $\Pi=\{p_1,p_2,p_3,p_4\}$ be the poset in Example~\ref{ex:oc}. 
The non-minimal elements are $p_3$ and $p_4$.  We choose the order $q_1=p_4$ and $q_2=p_3$. 
Then $w_{p_4}=-e_4$, $F_{p_4}=\conv\{-e_1,-e_2\}$, and $w_{p_3}=-e_3$, $F_{p_3}=\{-e_1\}$. 
Hence, 
\[ \calO(\Pi) \xrightarrow{\ \varphi_{p_4}\ } \calO_{\{p_4\}}(\Pi) \xrightarrow{\ \varphi_{p_3}\ } \calC(\Pi).\]
More explicitly, $\varphi_{p_4}$ changes only the fourth coordinate:
\[ \varphi_{p_4}(x_1,x_2,x_3,x_4) = \bigl(x_1,x_2,x_3, \min\{x_4-x_1,x_4-x_2\}\bigr),\]
and\[ \calO_{\{p_4\}}(\Pi) = \conv(\{ 0,e_2,e_3,e_4,e_3+e_4, e_1+e_3,e_2+e_3,e_1+e_2+e_3 \}).\]
Next,
\[ \varphi_{p_3}(x_1,x_2,x_3,x_4)= (x_1,x_2,x_3-x_1,x_4),\]
and hence
\[ \varphi_{p_3}\bigl(\calO_{\{p_4\}}(\Pi)\bigr) = \conv(\{ 0,e_1,e_2,e_3,e_4, e_1+e_2,e_2+e_3,e_3+e_4 \}) = \calC(\Pi).\]
Thus, in this example, Stanley's transfer map is explicitly realized as the composition of two combinatorial mutations.
\end{ex}


\bigskip

\begin{thebibliography}{99}
\bibitem{ACGK} M. Akhtar, T. Coates, S. Galkin and A. M. Kasprzyk, Minkowski polynomials and mutations, \textit{SIGMA Symmetry Integrability Geom. Methods Appl.} \textbf{8} (2012), 094, 17 pages. 
\bibitem{BR} M. Beck and S. Robins, ``Computing the Continuous Discretely'', Undergraduate Texts in Mathematics, Springer, 2007. 
\bibitem{CCGGK} T. Coates, A. Corti, S. Galkin, V. Golyshev and A. M. Kasprzyk, Mirror symmetry and Fano manifolds, In European Congress of Mathematics, pages 285--300. Eur. Math. Soc., Zurich, 2013. 
\bibitem{FF} X. Fang and G. Fourier, Marked chain-order polytopes, \textit{European J. Combin.} {\bf 58} (2016), 267--282. 
\bibitem{FFLP} X. Fang, G. Fourier, J.-P. Litza and C. Pegel, A continuous family of marked poset polytopes, \textit{SIAM J. Discrete Math.} {\bf 34} (2020), no.~1, 611--639.
\bibitem{FH} N. Fujita and A. Higashitani, Newton--Okounkov bodies of flag varieties and combinatorial mutations, \textit{Int. Math. Res. Not. IMRN}, no.12, (2021), 9567--9607. 
\bibitem{GU} S. Galkin and A. Usnich, Mutations of potentials, Preprint IPMU 10--0100, 2010. 
\bibitem{HL} T. Hibi and N. Li, Unimodular equivalence of order and chain polytopes, \textit{Math. Scand.} {\bf 118} (2016), no. 1, 5--12.
\bibitem{HN} A. Higashitani and Y. Nakajima, Deformations of dimer models, \textit{SIGMA Symmetry Integrability Geom. Methods Appl.}, {\bf 18} (2022), 030, 53 pages.
\bibitem{S86} R. P. Stanley, Two Poset Polytopes, \emph{Discrete Comput. Geom.} {\bf 1} (1986), 9--23. 
\bibitem{Sch} A. Schrijver,  ``Theory of Linear and Integer Programming'', John Wiley \& Sons, 1986. 
\end{thebibliography}
\end{document}